\documentclass[12pt, a4paper]{amsart}
\usepackage{amsmath}
\usepackage{geometry,amsthm,graphics,tabularx,amssymb,shapepar}
\usepackage{latexsym,amsfonts,amssymb,amsmath,amsxtra,bbm,color,mathrsfs}
\usepackage{amscd}
\usepackage[hidelinks]{hyperref}
\usepackage{pdfsync}
\usepackage{braket}
\usepackage{tikz-cd}
\usepackage[all]{xypic}

\theoremstyle{definition}
\newtheorem*{ack}{Acknowledgements}

\long\def\delete#1{}

\newcommand{\BC}{{\mathbb {C}}}

\newcommand{\BH}{{\mathbb {H}}}

\newcommand{\CK}{{\mathcal {K}}}

\newcommand{\CO}{{\mathcal {O}}}

\newcommand{\GL}{{\mathrm{GL}}}

\renewcommand{\Im}{{\mathrm{Im}}}
\newcommand{\Ind}{{\mathrm{Ind}}}

\newcommand{\Lie}{{\mathrm{Lie}}}

\newcommand{\pr}{{\mathrm{pr}}}

\newcommand{\PGL}{{\mathrm{PGL}}}

\newcommand{\Res}{{\mathrm{Res}}}

\newcommand{\SL}{{\mathrm{SL}}}

\newcommand{\SU}{{\mathrm{SU}}}

\newcommand{\vsp}{{\vspace{0.2in}}}

\newcommand{\sgn}{\operatorname{sgn}}

\newcommand{\g}{\mathfrak g}

\newcommand{\n}{\mathfrak n}

\renewcommand{\l}{\mathfrak l}

\newcommand{\s}{\mathfrak s}

\newcommand{\C}{\mathbb{C}}
\newcommand{\R}{\mathbb R}

\newcommand{\abs}[1]{\lvert#1\rvert}

\newcommand{\be}{\begin {equation}}
\newcommand{\ee}{\end {equation}}
\newcommand{\bee}{\begin {equation*}}
\newcommand{\eee}{\end {equation*}}

\theoremstyle{plain}
\newtheorem{theorem}{Theorem}[section]
\newtheorem{thm}[theorem]{Theorem}
\newtheorem{cor}[theorem]{Corollary}

\newtheorem{lem}[theorem]{Lemma}

\newtheorem{prpt}[theorem]{Proposition}

\theoremstyle{definition}
\newtheorem{rem}[theorem]{Remark}
\newtheorem{remark}[theorem]{Remark}

\def\bfmm{\mathbf{m}}
\def\Unip{{\mathrm{Unip}}}

\def\bfone{\mathbf{1}}
\def\bfrr{\mathbf{r}}

\def\bC{{\mathbb C}}
\def\bZ{{\mathbb Z}}
\def\ckG{{\check G}}
\def\cO{{\mathcal O}}
\def\ckcO{{\check \cO}}
\def\ckcOe{\ckcO_{\mathrm e}}
\def\ckcOo{\ckcO_{\mathrm o}}

\def\SU{\mathrm{SU}}
\def\U{\mathrm{U}}
\def\tU{\widetilde{\mathrm{U}}}

\def\SL{\mathrm{SL}}
\def\PGL{\mathrm{PGL}}
\def\GL{\mathrm{GL}}

\def\bR{\mathbb{R}}
\def\AND{\quad \text{and}\quad}
\def\PGL{{\mathrm{PGL}}}

\def\Coh{\mathrm{Coh}}
\def\Irr{\mathrm{Irr}}

\def\sfS{\mathsf{S}}

\def\sfW{\mathsf{W}}
\def\half{\frac{1}{2}}

\def\cC{\mathcal{C}}
\def\cK{\mathcal{K}}
\def\cCb{\mathcal{C}^{\mathrm b}}
\def\cCg{\mathcal{C}^{\mathrm g}}
\def\cCd{\mathcal{C}^{\mathrm d}}
\def\fhh{\mathfrak{h}}

\def\bfaa{\mathbf{a}}
\newcommand{\trivial}[2][]{\if\relax\detokenize{#1}\relax
{\color{red!50!black}\vspace{0em}{[} #2 {]}}\else
\ifx#1h
\ifcsname showtrivial\endcsname
{\color{red!20!black}\vspace{0em}{[}  #2 {]}}\fi
\else {\red Wrong argument!} \fi
\fi
}

\begin{document}

\title[Special unipotent representations]{Special unipotent representations of simple linear Lie groups of type $A$}

\author [D. Barbasch] {Dan Barbasch}
\address{Department of Mathematics\\
Cornell University\\
Ithaca, NY14853, USA}
\email{barbasch@math.cornell.edu}

\author [J.-J. Ma] {Jia-Jun Ma}
\address{School of Mathematical Sciences\\
  Xiamen University
  Xiamen, 361005, China}
  \address{Department of Mathematics, Xiamen University Malaysia campus, Sepang, Selangor Darul Ehsan, 43900,  Malaysia}
 \email{hoxide@xmu.edu.cn}

\author [B. Sun] {Binyong Sun}
\address{Institute for Advanced Study in Mathematics \& New Cornerstone Science Laboratory, Zhejiang University,  Hangzhou, 310058, China}\email{sunbinyong@zju.edu.cn}

\author [C.-B. Zhu] {Chen-Bo Zhu}
\address{Department of Mathematics\\
  National University of Singapore\\
  10 Lower Kent Ridge Road, Singapore 119076} \email{matzhucb@nus.edu.sg}



\begin{abstract} Let $G$ be a special linear group over the real, the complex or the quaternion, or a special unitary group. In this note, we determine all special unipotent representations of $G$ in the sense of Arthur and Barbasch-Vogan, and show in particular that all of them are unitarizable.
\end{abstract}

\maketitle

\section{Introduction}

Let $G_\BC$ be a connected  reductive complex Lie group, and let $G$ be a real form of $G_\BC$, namely the fixed point group of an anti-holomorphic involutive automorphism $\sigma$ of $G_\BC$.  Denote by $\check G$ the Langlands dual group of $ G$ (which is a connected reductive complex Lie group), and by $\check \g$ the Lie algebra of $\check G$. We will work in the category of Casselman-Wallach representations \cite[Chapter 11]{Wa2}. For  a nilpotent $\check G$-orbit $\check \CO$ in $\check \g$, let $\mathrm{Unip}_{\check \CO}(G)$ be the set of isomorphism classes of special unipotent representations of $G$ attached to $\check \CO$. See \cite[Definition 5.23]{BVUni} for the definition of special unipotent, and \cite[Section 2]{BMSZ1} for a comprehensive discussion. The Arthur-Barbasch-Vogan conjecture (\cite[Section 4]{ArUni}, \cite[Introduction]{ABV}) asserts that all representations in $\mathrm{Unip}_{\check \CO}(G)$ are unitarizable. It is easy to see that the conjecture is reduced to the case when $G_\BC$ is simply connected, and the (real) Lie algebra $\mathrm{Lie}(G)$ of $G$ is simple.  In this case the Lie algebra of $G_\BC$, denoted by $\g$, is either simple or the product of two isomorphic simple Lie algebras.

In a series of two papers \cite{BMSZ1,BMSZ2}, the authors construct and classify special unipotent representations of real classical groups. We refer the reader to the introductory sections of \cite{BMSZ1,BMSZ2} for the list of real classical groups covered in the classifications. The main ingredients consist of Kazhdan-Lusztig theory (as in the work of Lusztig, Joseph, and Barbasch-Vogan), Howe's theory of theta lifting \cite{Howe79,Howe89} and Vogan's theory of associated cycles \cite{Vo89}.
As a direct consequence of the construction and the classification, the authors show that all special unipotent representations of real classical groups are unitarizable, as predicted by the Arthur-Barbasch-Vogan conjecture. For quasi-split classical groups, the unitarity is independently established in \cite{AAM,AM}, from the perspective of the endoscopic classification of representations \cite{ArEnd, Mok}.

In another paper \cite{BMSZ3}, the authors consider Spin groups and determine all genuine special unipotent representations. In particular, it is shown that all of them are unitarizable.

The results of \cite{BMSZ1,BMSZ2, BMSZ3} therefore establish the validity of the Arthur-Barbasch-Vogan conjecture for all $G$ when
$G_\BC$ is simply connected, and $\mathrm{Lie}(G)$ is simple, and $G$ is of type $B,C,D$. Note that for a complex classical group $G$ of type $B,C,D$, the unitarity assertion was due to Barbasch \cite{B89}, as part of his classification of the unitary duals for these groups.

In this note we are concerned with the case when $G$ is of type $A$, namely $G$ is one of the following: ($n\geq 2$)
\begin{equation}\label{gptypea}
   G=\SL_n(\R), \ \SL_{\frac{n}{2}}(\BH) \ (n\textrm{ is even}), \ \mathrm{SU}(p,q) \ (p+q=n),  \quad\textrm{ or }\quad \SL_n(\BC).
\end{equation}
The complex Lie algebra $\g$ is
\[
\g=\s\l_n(\BC), \ \s\l_n(\BC), \ \s\l_n(\BC), \quad\textrm{ or }\quad \s\l_n(\BC)\times {\s\l_n(\C)},
\]
and the complex Lie group $G_\BC$ is
\[
G_\BC=\SL_n(\BC), \ \SL_n(\BC),  \ \SL_{n}(\BC),  \quad\textrm{ or }\quad \SL_n(\BC)\times {\SL_n(\BC)}.
\]
Here the complexfication of  $\s\l_n(\bC)$ is identified with $\s\l_n(\BC)\times {\s\l_n(\C)}$ via the complexification map
\begin{equation}\label{complex}
{\s\l_n(\C)}\rightarrow \s\l_n(\BC)\times {\s\l_n(\C)}, \quad x\mapsto (x, \bar x),
\end{equation}
where $\bar x$ denotes the entry-wise complex conjugation of $x$. Similarly $\SL_n(\bC)$ is identified with  a real form of  $\SL_n(\BC)\times {\SL_n(\C)}$. Likewise we will identify  $\GL_n(\BC)$ with  a real form of $\GL_n(\BC)\times {\GL_n(\C)}$.

The list in \eqref{gptypea} exhausts all the groups $G$ with the following properties: $G_\BC$ is simply connected, $\Lie(G)$ is simple, and $\g$ is a product of simple Lie algebras of type $A$.

For a group $G$ in \eqref{gptypea}, define respectively a group
\begin{equation}\label{eq:G'}
 G':=\GL_n(\R), \ \GL_{\frac{n}{2}}(\BH), \ \widetilde{ \mathrm{U}}(p,q),\ \quad\textrm{ or }\quad \GL_n(\BC)
\end{equation}
so that $G$ is naturally identified with a subgroup of $G'$. Here $\tU(p,q)$ denotes the double cover of $\mathrm{U}(p,q)$ defined by the square root of the determinant character:
\[\tU(p,q) = \set{(u,t)\in \U(p,q)\times \U(1) | \det u = t^2},\]
and $\SU(p,q)$ is identified as the subgroup $\SU(p,q)\times \set{1}\subset \tU(p,q)$.  Also define respectively
\[
   G'_\BC:=\GL_n(\BC), \ \GL_{n}(\BC),  \ \widetilde \GL_n(\C),  \quad\textrm{ or }\quad \GL_n(\BC)\times {\GL_n(\BC)},
\]
where $\widetilde \GL_n(\C)$ denotes the double cover of $\GL_n(\C)$ defined likewise by the square root of the determinant character. Then $G'$ is a real form of $G'_\BC$, and $G_\BC$ is naturally identified with a subgroup of $G'_\BC$.

Write $\check G'$ for the Langlands dual group of $G'_\BC$, with the  Lie algebra $\check \g'$. Then $\check \g$ is a Lie subalgebra of $\check \g'$, and $\check \CO$ is also a nilpotent $\check G'$-orbit in $\check \g'$. We will determine all representations in $\mathrm{Unip}_{\check \CO}(G)$ via the restriction of representations in $\mathrm{Unip}_{\check \CO}(G')$. Note that all representations in $\mathrm{Unip}_{\check \CO}(G')$ have been classified, and are known to be unitarizable. See \cite{V.GL} for general linear groups and \cite[Section 11.2]{BMSZ2} for
$\widetilde{ \mathrm{U}}(p,q)$. See also \cite[Sections 2.4-2.5]{BMSZ1} for a review of their classifications.

 Recall that by Clifford theory, for every irreducible Casselman-Wallach representation $\pi$ of $\GL_n(\R)$,  if $\pi\otimes \sgn\cong \pi$, then $\pi|_{\SL_n(\R)}$ is the direct sum of two irreducible subrepresentations that are not isomorphic, to be denoted by $\pi|_{\SL_n(\R)}^+$ and $\pi|_{\SL_n(\R)}^-$.  Here $\sgn$ denotes the unique non-trivial quadratic character of $\GL_n(\R)$. If $\pi\otimes \sgn\ncong \pi$,  then $\pi|_{\SL_n(\R)}$ is irreducible.

\begin{thm}\label{main}
(a) Suppose that $G=\SL_n(\R)$ so that $G'=\GL_n(\R)$.
Define
\[\mathrm{Unip}_{\check \CO}(G')_0=
\{\pi\in \mathrm{Unip}_{\check \CO}(G')\,:\,  \pi\otimes \sgn\cong \pi \}
\]
and fix a decomposition
\[
\mathrm{Unip}_{\check \CO}(G')\setminus \mathrm{Unip}_{\check \CO}(G')_0 =\mathrm{Unip}_{\check \CO}(G')_+\sqcup \mathrm{Unip}_{\check \CO}(G')_-
\]
such that
\[
\mathrm{Unip}_{\check \CO}(G')_-=
\{\pi\otimes \sgn\,:\,  \pi\in \mathrm{Unip}_{\check \CO}(G')_+\}.
\]
Then
\[
\mathrm{Unip}_{\check \CO}(G)=\{\pi|_G\,:\, \pi\in \mathrm{Unip}_{\check \CO} (G')_+ \}\sqcup \bigsqcup_{\pi\in \mathrm{Unip}_{\check \CO} (G')_0 } \{\pi|_G^+, \pi|_G^-\}.
\]
\noindent
(b) Suppose that $G=\SL_{\frac{n}{2}}(\BH)$  ($n$ is even), $\mathrm{SU}(p,q)$ ($p+q=n$), or $\SL_n(\BC)$. Then for every $\pi\in \mathrm{Unip}_{\check \CO}(G')$, the representation $\pi|_G$ is irreducible and belongs to $\mathrm{Unip}_{\check \CO}(G)$. Moreover,  the map
\[
\mathrm{Unip}_{\check \CO}(G')\rightarrow \mathrm{Unip}_{\check \CO}(G), \quad \pi\mapsto \pi|_G
\]
is a bijection.

\end{thm}

\begin{rem}
Suppose that $G'=\GL_n(\BC)$. If $\check \CO$ has the form $\ckcO'\times {\ckcO'}$, where $\ckcO'$ is a nilpotent $\GL_n(\C)$-orbit in $\g\l_n(\C)$,
then $\mathrm{Unip}_{\ckcO}(G')$ is a singleton. Otherwise the set $\Unip_{\ckcO}(G')$ and $\Unip_{\ckcO}(G)$ are both empty. See \cite[Section 5]{BVUni} and \cite{Vo89}.
\end{rem}

\begin{cor}
Suppose that $G=\SL_n(\R)$, $\SL_{\frac{n}{2}}(\BH)$  ($n$ is even), $\mathrm{SU}(p,q)$ ($p+q=n$),  or $\SL_n(\BC)$. Then all representations in  $\mathrm{Unip}_{\check \CO}(G)$ are unitarizable.
\end{cor}

\begin{remark} For $G=\SL_n(\BC)$, the result follows easily from those of \cite{V.GL} or \cite{B89}.
\end{remark}

Theorem \ref{main} gives the description of $\mathrm{Unip}_{\check \CO}(G)$
in terms of $\mathrm{Unip}_{\check \CO}(G')$. Our proof of Theorem \ref{main} is simple, and in the case of $G=\mathrm{SU}(p,q)$ ($p+q=n$), or $\SL_n(\C)$, it depends on the following counting result. Here the basic idea is that one can count irreducible representations from the coherent continuation representation of the integral Weyl group. The idea first appeared in \cite{BV.W} and is developed in its full generality in \cite{BMSZ1}.

\begin{prpt}\label{main3}
Suppose that $G=\mathrm{SU}(p,q)$ ($p+q=n$), or $\SL_n(\C)$. Then
\[
  \sharp \mathrm{Unip}_{\check \CO}(G') =\sharp \mathrm{Unip}_{\check \CO}(G).
  \]
\end{prpt}

 Here and henceforth $\sharp$ indicates the cardinality of a finite set.

\section{Proof of Theorem \ref{main}}\label{sec2}
Note that for every $\pi\in \mathrm{Unip}_{\check \CO}(G')$, $\pi|_G$ is either a representation in $\mathrm{Unip}_{\check \CO}(G)$, or the direct sum of two distinct  representations in $\mathrm{Unip}_{\check \CO}(G)$.

Write $Z^\circ$ for the connected component of the identity of the center of $G'$. The constraint on the infinitesimal character imposed by ``special unipotent'' implies that  $Z^\circ$ acts trivially on every representation in $\mathrm{Unip}_{\check \CO}(G')$.

If $G=\SL_{\frac{n}{2}}(\BH)$  ($n$ is even), then $G'=G \times Z^\circ$ is a direct product, and hence the map
\begin{equation}\label{rest0}
\mathrm{Unip}_{\check \CO}(G')\rightarrow \mathrm{Unip}_{\check \CO}(G), \quad \pi\mapsto \pi|_G
\end{equation}
is bijective. Part (b) of Theorem \ref{main} follows in this case.

If $G=\mathrm{SU}(p,q)$ ($p+q=n$),
or $\SL_n(\BC)$, then $G'=G Z^\circ$, and the map \eqref{rest0}
is well-defined and injective. Therefore part (b) of Theorem \ref{main} follows from Proposition \ref{main3}.

This completes the proof of part (b) of Theorem \ref{main} modulo the proof of Proposition \ref{main3}. We discuss part (a) below.

Assume $G=\SL_n(\R)$. Then $Z^\circ =\R^\times_+$ (the group of positive real numbers), and $\GL_n^+(\R)=Z^\circ \times G$, where $\GL_n^+(\R)$ is the connected component of the identity of $G'=\GL_n(\R)$. Define the set $\mathrm{Unip}_{\check \CO}(\GL_n^+(\R))$ of special unipotent representations of $\GL_n^+(\R)$ attached to $\check \CO$,
in the obvious way. The map
\[
\mathrm{Unip}_{\check \CO}(\GL_n^+(\R))\rightarrow \mathrm{Unip}_{\check \CO}(G), \quad \pi\mapsto \pi|_G
\]
is well-defined and bijective, and hence part (a) of Theorem \ref{main} follows by Clifford theory. \qed

We supplement Part (a) of Theorem \ref{main} with an explicit description of special unipotent representations of $\SL_n(\bR)$, by giving an explicit decomposition
\[
\mathrm{Unip}_{\check \CO}(\GL_n(\bR))=\mathrm{Unip}_{\check \CO}(\GL_n(\bR))_+\sqcup \mathrm{Unip}_{\check \CO}(\GL_n(\bR))_-\sqcup \mathrm{Unip}_{\check \CO}(\GL_n(\bR))_0.\]

We identify $\ckcO$ with the corresponding Young diagram \cite{CM}.
Write the lengths of the non-zero rows of $\ckcO$ as a multiset
\[
\bfrr(\ckcO) := \set{\underbrace{r_1, \cdots, r_1}_{m_1 \text{ terms}},
\underbrace{r_2, \cdots, r_2}_{m_2 \text{ terms}},
\cdots,
\underbrace{r_k, \cdots, r_k}_{m_k \text{ terms}}
}
\]
with $r_1>r_2>\cdots > r_k >0$. Put
\[
D_\ckcO  := \Set{(a_1,a_2, \cdots, a_k) \in \mathbb Z^k  | \,0\leq a_l \leq m_l, \text{ for $1\leq l \leq k$}}.\]
For a tuple $\bfaa:=(a_1,a_2,\cdots, a_k)\in D_\ckcO$,
define the normalized induced representation
\[
\pi_\bfaa := \Ind_{P}^{\GL_n(\bR)} \bigotimes_{l=1}^k (\bfone_{r_l}^{\otimes (m_l-a_l)}  \otimes\sgn_{r_l}^{\otimes a_l}).
\]
Here $P$ is the standard parabolic subgroup of $\GL_n(\bR)$ with Levi component
\[
\prod_{l=1}^{k}
\underbrace{\GL_{r_l}(\bR) \times \cdots \times \GL_{r_l}(\bR)}_{m_l \text{ terms}}, \]
and $\bfone_{r_l}$ (resp. $\sgn_{r_l}$) denotes the trivial (rep. sign) character of $\GL_{r_l}(\bR)$. Then
\[\mathrm{Unip}_{\check \CO}(\GL_n(\bR))=
\{\pi _\bfaa \,:\,  \bfaa \in D_\ckcO\}.\]
In words, the modules are induced irreducible from the parabolic subgroup with Levi component of the form above, and the representations are all choices of trivial or sign on each factor $\GL_{r_l}(\bR)$ of Levi component. See \cite{V.GL} and \cite[Example~27.5]{ABV}.

It is clear from the construction that $\pi_\bfaa \otimes \sgn \cong \pi_{\mathbf m-\mathbf a}$, where $\bfmm :=(m_1,\cdots, m_k)$.
In particular
$\pi_\bfaa \otimes \sgn \cong \pi_{\mathbf{a}}$ if and only if $\mathbf a = \frac{\mathbf m}{2}$. Therefore, we conclude that
\[\mathrm{Unip}_{\check \CO}(\GL_n(\bR))_0=
\{\pi _\bfaa \,:
2\bfaa = \bfmm
\}.
\]
This is a singleton if every $m_l$ ($1\leq l\leq k$) is  even, and is empty otherwise. We may also take
\[\mathrm{Unip}_{\check \CO}(\GL_n(\bR))_+=
\{\pi _\bfaa \,:\,  \bfaa \in D_\ckcO,  2\bfaa<\mathbf m \}.
\]
Here ``$ <$'' indicates the lexicographic order on $\bZ^k$.

\section{Proof of Proposition \ref{main3} for $\SL_n(\bC)$}\label{sec:SL}

\def\halfn{\frac{n}{2}}
\def\tX{\widetilde{X}}

\def\tpr{\widetilde{\pr}}
\def\trho{\widetilde{\rho}}
\def\tG{\widetilde{G}}
\def\tB{\widetilde{B}}
\def\tK{\widetilde{K}}
\def\tP{\widetilde{P}}
\def\barP{\overline{P}}

We will use the coherent continuation representation to count the special unipotent representations, as in \cite{BMSZ1}. We adopt its formulation as \cite[Sections 3 and 4] {BMSZ1}. The original references include \cite{Sch, Zu, SpVo, Vg}.

We have $G = \SL_n(\bC)$ and $G'= \GL_n(\bC)$. We make the following identifications:
\begin{itemize}
\item The dual $\fhh'^*$ of the abstract Cartan subalgebra $\fhh'$ of $\g'$ is identified with $\bC^n \times {\bC}^n$, and the dual $\fhh^*$ of the abstract Cartan subalgebra $\fhh$ of $\g$ is identified with the quotient $\bC^n/\bC\bfone_n\times {\bC}^n/{\bC}\bfone_n$, where $\bfone_n:= (1, 1, \cdots, 1)$.
\item The analytic weight lattice $X'\subset \fhh'^*$ of $G'_\bC$ is identified with $\bZ^n\times \bZ^n$ and the analytic weight lattice $X \subset \fhh^*$ of $G_\bC$ is identified with the quotient
$\bZ^n/\bZ \bfone_n \times \bZ^n/\bZ\bfone_n$.
\item The abstract Weyl group $W'$ of $G'_\bC$ and the abstract Weyl group $W$ of $G_\BC$ are naturally identified with $\sfS_{n}\times
\sfS_{n}$.
\end{itemize}
Here and henceforth, $\sfS_{n}$ denotes the symmetric group in $n$ letters.

Fix a $X'$-coset $\Lambda' \subset  \fhh'^*$.
We have the so-called integral Weyl group
\[
W'(\Lambda') := \set{w'\in W' |
\text{$w' \lambda' -\lambda'$ is in the root lattice for every $\lambda'\in \Lambda'$}
}.
\]

Let $\mathcal{K} (\GL_n(\bC))$ be the
Grothendieck group (with coefficients in $\bC$) of the category of Casselman-Wallach representations of  $\GL_n(\bC)$. Denote by $\Coh_{\Lambda'}(\CK(\GL_n(\bC)))$ the space of $\mathcal{K} (\GL_n(\bC))$-valued coherent families based on $\Lambda'$.
The space $\Coh_{\Lambda'}(\CK (\GL_n(\bC)))$
is naturally a $W'(\Lambda')$-module, which is called the coherent continuation representation. (See \cite[Sections 3 and 4]{BMSZ1} for more details.)
Similar notations will be used without further explanation.

For an element $\lambda'\in \fhh'^*$, let $[\lambda']\in \fhh^*$ denote its image  under the natural map $\fhh'^*\rightarrow \fhh^*$.  Likewise for an $X'$-coset $\Lambda'=\lambda '+X' \subset  \fhh'^*$, we have the image $\Lambda = [\lambda ']+X \subset  \fhh^*$.

The restriction of representations from $\GL_n(\bC)$ to $\SL_n(\bC)$
induces a homomorphism
\[
\mathcal{K} (\GL_n(\bC))\rightarrow \mathcal{K} (\SL_n(\bC)), \quad \Pi\mapsto \Pi|_{\SL_n(\bC)}.
\]
One checks that there  is a unique map
\[
    \Res\colon  \Coh_{\Lambda'}(\CK (\GL_n(\bC)))\rightarrow \Coh_{\Lambda}(\CK (\SL_n(\bC)))
\]
such that for each $\Psi\in \Coh_{\Lambda '}(\CK (\GL_n(\bC)))$,
\[
\Res(\Psi)([\lambda']) = \Psi(\lambda')|_{\SL_n(\bC)}
\quad \text{for all $\lambda'\in \Lambda'$.}
\]

Let $\n_h$ and $\n_0$ are two natural numbers such that $\n_h+\n_0 = n$.
Consider the $X'$-coset
\begin{equation}\label{eq:Lambda'}
\Lambda' = ( (\underbrace{\half, \cdots, \half}_{\n_h},
\underbrace{0,\cdots,0}_{\n_0})+ \bZ^n)
\times
( (\underbrace{\half, \cdots, \half}_{\n_h},
\underbrace{0,\cdots,0}_{\n_0})+ \bZ^n)
\subset \fhh'^*.
\end{equation}

The integral Weyl group $W(\Lambda')$ for $G'=\GL_n(\bC)$ and
the integral Weyl group $W(\Lambda)$ for $G=\SL_n(\bC)$ are naturally identified and are isomorphic to
\[
\sfS_{\n_h}\times \sfS_{\n_0}\times \sfS_{\n_h} \times \sfS_{\n_0}\subset \sfS_{n}\times
\sfS_{n}.
\]

We introduce some notations. For a finite group $E$, denote by $\Irr (E)$ the set of isomorphism classes of irreducible representations of $E$; $[\ : \ ]$ indicates the multiplicity of the first (irreducible) representation in the second representation.

It is well-known and straightforward to check that
\[
 \Coh_{\Lambda'}(\CK(\GL_n(\bC)))
 \cong \bigoplus_{\sigma_1\in \Irr(\sfS_{\n_h})}
     \bigoplus_{\sigma_2\in \Irr(\sfS_{\n_0})}\sigma_1\otimes \sigma_2\otimes \sigma_1\otimes \sigma_2.
\]
Namely it is isomorphic to the regular representation of $\sfS_{\n_h}\times \sfS_{\n_0}$ (see \cite[page 55]{BVUni}). The following lemma is then routine to check by using
Zhelobenko's classification for irreducible representations of complex groups (see for example \cite[Introduction]{BVUni}).

\begin{lem}\label{lem:cohSLR}
    The map
    \[
    \Res\colon  \Coh_{\Lambda'}(\CK(\GL_n(\bC)))\rightarrow \Coh_{\Lambda}(\CK(\SL_n(\bC)))
    \]
    is an injective homomorphism of $W(\Lambda)$-modules. Furthermore,
    $\Res$ is an isomorphism if $\n_h\neq \n_0$. If $\n_h =  \n_0$, we have that
   \[
   \Coh_{\Lambda}(\CK(\SL_n(\bC)))
    \cong  \Im(\Res) \oplus \bigoplus_{\sigma_1\in \Irr(\sfS_{\n_h})}
     \bigoplus_{\sigma_2\in \Irr(\sfS_{\n_0})}(\sigma_1\otimes \sigma_2)\otimes (\sigma_2\otimes \sigma_1),
   \]
where $\Im(\Res)$ denotes the image of the map $\Res$.
\end{lem}
\def\sftt{\mathsf{t}}
\begin{remark}
Suppose $\n_h=\n_0$. The stabilizer $W_\Lambda$ of $\Lambda$ in $W =
\sfS_{n}\times
\sfS_{n}$ is generated by $W(\Lambda)$ together with $(\sftt, 1)$ and $(1,\sftt)$, where
$\sftt$ is the involution in $\sfS_n$ switching $i$ and $n-i+1$ for all $i\in \set{1,2,\cdots,n}$. The additional term in the coherent continuation representation $\Coh_{\Lambda}(\CK(\SL_n(\bC)))$ arises due to the fact that $W(\Lambda)$ is a proper subgroup of $W_\Lambda$.
\end{remark}

Given two Young diagrams $\imath$ and $\jmath$,  write $\imath\overset{r}{\sqcup}\jmath$ for the Young diagram whose multiset of nonzero row lengths equals the union of those of $\imath$ and $\jmath$. Also denote by $|\imath|$ the total size of $\imath$.

\vsp

\begin{proof}[Proof of Proposition~\ref{main3} for $\SL_n(\bC)$]
By the remark after Theorem~\ref{main}, it suffices to consider the case when $\ckcO = \ckcO' \times \ckcO'$, where $\ckcO'$ is a nilpotent $\GL_n(\C)$-orbit in $\g\l_n(\C)$.
Write the Young diagram decomposition $\ckcO' = \ckcO'_e\overset{r}{\sqcup}\ckcO'_o$, where all nonzero row lengths of $\ckcO'_e$ (resp. $\ckcO'_o$) are even (resp. odd).
The infinitesimal character for $G'$ determined by  $\ckcO$ (which is an algebraic character of the center of the universal algebra of $\g':=\Lie(G'_\BC$)), as in \cite[Section 5]{BVUni}, is represented by an element $\lambda'_{\check \CO}\in \Lambda'$
with $\n_h= \abs{\ckcO'_e}$ and $\n_0=\abs{\ckcO'_o}$ (see \eqref{eq:Lambda'}).
The infinitesimal character for $G$ determined by $\ckcO$ is represented by $[\lambda'_\ckcO] \in \Lambda$.
It is well-known (see for example \cite[Section 7.1]{BMSZ1}) that the Lusztig left cell attached to $\ckcO$ (\cite[Section 7]{BMSZ1}) is a singleton consisting of $\sigma_{\ckcO}:=((\ckcO'_e)^{\mathrm t}\otimes (\ckcO'_o)^{\mathrm t})\otimes ((\ckcO'_e)^{\mathrm t}\otimes (\ckcO'_o)^{\mathrm t})$. Here a superscript ``$\mathrm t$'' indicates the transpose of a Young diagram, and for any natural number $k$, we identify an element of $\Irr(\sfS_{k})$ with a Young diagram of total size $k$ via the Springer
correspondence, as in \cite[11.4]{Carter}. When $\n_h =  \n_0$, in view of the fact that $(\ckcO'_e)^{\mathrm t}\ne (\ckcO'_o)^{\mathrm t}$, we have that
\[[\sigma_{\ckcO}: \bigoplus_{\sigma_1\in \Irr(\sfS_{\n_h})}
     \bigoplus_{\sigma_2\in \Irr(\sfS_{\n_0})}(\sigma_1\otimes \sigma_2)\otimes (\sigma_2\otimes \sigma_1)]=0.
\]
Applying Lemma~\ref{lem:cohSLR}, we have that
\[
[\sigma_{\ckcO}:  \Coh_{\Lambda}(\SL_n(\bC))   ]  =
[\sigma_{\ckcO}:  \Coh_{\Lambda'}(\GL_n(\bC)) ].     \]
The result then follows by the counting equality of special unipotent representations in terms of the coherent continuation representation \cite[Corollary 2.2]{BMSZ1}.
\end{proof}

\section{Proof of Proposition \ref{main3} for $\SU(p,q)$}\label{sec:SU}

We adopt similar notations and terminologies as Section \ref{sec:SL}.

In this section, $G = \SU(p,q)$ and $G'= \tU(p,q)$, where $p+q=n$. The nilpotent orbits of $\ckG = \PGL_n(\bC)$ and $\check G'=\GL_{n}(\C)/\{\pm 1\}$ are both parameterized by partitions of $n$. We will identify the two sets.

Recall that $G_\bC = \SL_n(\bC)$, and $G'_\BC = \widetilde \GL_n(\C)$.
We make the following identifications:
\begin{itemize}
\item The dual $\fhh'^*$ of the abstract Cartan subalgebra $\fhh'$ of $\g'$ is identified with $\bC^n$, and the dual $\fhh^*$ of the abstract Cartan subalgebra $\fhh$ of $\g$ is identified with the quotient $\bC^n/\bC\bfone_n$.
\item
Write $X'\subset \fhh'^*$ for the lattice of weights occurring in finite dimensional representations of $G'_\bC$ which factor through algebraic representations of $\GL_n(\bC)$.
We identify $X'$ with $\bZ^n\subset \fhh'^*$ and the analytic weight lattice $X \subset \fhh^*$ of $G_\bC$ is then identified with the quotient
$\bZ^n/\bZ \bfone_n $.
\item The abstract Weyl group $W'$ of $G'_\bC$ and the abstract Weyl group $W$ of $G_\BC$ are naturally identified with $\sfS_{n}$.
\end{itemize}

Let $\n_h$ and $\n_0$ be two natural numbers such that $\n_h+\n_0 = n$.
Consider the $\bZ^n$-coset
\begin{equation}\label{eq:Lambda'.U}
\Lambda'=  (\underbrace{\half, \cdots \half}_{\n_h\text{-terms}},\underbrace{0, \cdots 0}_{\n_0\text{-terms}}) + \bZ^n \subset
\fhh'^*,
\end{equation}
and denote by $\Lambda $ the image of $\Lambda'$ under the natural map $\fhh'^*\rightarrow \fhh^*$.

The integral Weyl group $W(\Lambda')$ for $G'$ and
the integral Weyl group $W(\Lambda)$ for $G$ are naturally identified and are isomorphic to $\sfS_{\n_h}\times \sfS_{\n_0}$.

Let $\sfW_n$ be the subgroup of $\sfS_{2n}$ centralizing all transpositions of the form $(k,2n-k+1)$ ($1\leq k\leq n$).
The group $\sfW_n$ is isomorphic to the  Weyl group of type $B_n$ (or type $C_n$).
We  further introduce some notations for Weyl group representations.
For natural numbers $r$, $p$ and $q$,   define
\[
\begin{split}
\cCb_{r} &:=
\Ind_{\sfW_{r}}^{\sfS_{2r}} \bfone ,\\
\cCg_{p,q} &:= 
    \bigoplus_{0\leq k \leq \min(p,q)}\Ind_{\sfW_{k}\times \sfS_{p-k}\times \sfS_{q-k}}^{\sfS_{p+q}}\bfone\otimes \sgn\otimes \sgn, \AND \\
    \cCd_r &:= \Ind_{\sfS_r}^{\sfS_r\times \sfS_r} \bfone.
\end{split}
\]
Here $\bfone$ (resp. $\sgn$) denotes the trivial (resp. sign) character of an appropriate Weyl group. We also define
\[
\begin{split}
\cC^1_{p,q,r} & := \begin{cases}
    \cCb_{r/2}\otimes \cCg_{p-r/2,q-r/2} & \text{if $r$ is even and $\min(p,q)\geq r/2$};\\
    0 & \text{otherwise,}
\end{cases}
\\
\cC^2_{p,q,r} & := \begin{cases}
     \cCg_{p-r/2,q-r/2}\otimes \cCb_{r/2} & \text{if $r$ is even and $\min(p,q)\geq r/2$};\\
    0 & \text{otherwise.}
\end{cases}
\end{split}
\]

The following lemma can be deduced from a result of Barbasch-Vogan (see \cite[Theorem~4.5]{BMSZ1}) by direct computation, similar to \cite[Propositions 8.1 and 8.2]{BMSZ1}.

\begin{lem}\label{lem:cohSUpq}
As a $W(\Lambda)$-module,
   $\Coh_{\Lambda}(\CK (\SU(p,q)))$ is isomorphic to
\[
\begin{cases}
    \cC^1_{p,p,p} \oplus \cC^2_{p,p,p} \oplus
    \cCd_p\oplus \cCd_p, & \text{if $p=q=\n_h = \n_0$};\\
     \cC^1_{p,q,\n_h} \oplus \cC^2_{p,q,\n_0}  , & \text{otherwise}.
\end{cases}
\]
\end{lem}

\vsp

\begin{proof}[Proof of Proposition~\ref{main3} for  $\SU(p,q)$]
Write the Young diagram decomposition $\ckcO = \ckcO_e\overset{r}{\sqcup} \ckcO_o$, where all nonzero row lengths of $\ckcO_e$ (resp. $\ckcO_o$) are even (resp. odd).
 The infinitesimal character for  $G'$  determined by $\ckcO$, as in \cite[Section 5]{BVUni}, is represented by an element $\lambda'_{\ckcO}\in \Lambda'$
with $\n_h= \abs{\ckcO_e}$ and $\n_0=\abs{\ckcO_o}$  (see \eqref{eq:Lambda'.U}).
The infinitesimal character for $G$ determined by $\ckcO$ is represented by $[\lambda'_\ckcO]\in \Lambda$.
The Lusztig left cell attached to $\ckcO$ (see \cite[Section 7]{BMSZ1}) is a singleton consisting of $\sigma_{\ckcO}:=(\ckcOe)^{\mathrm t} \otimes (\ckcOo)^{\mathrm t}$.
Since $\cCd_p = \bigoplus_{\sigma\in \Irr(\sfS_p)} \sigma\otimes \sigma$ and since
$\ckcOe^{\mathrm t} \neq \ckcOo^{\mathrm t}$,
we have
\[   [\sigma_{\ckcO} : \cCd_p] = 0, \quad \text{if $p=q=\n_h = \n_0$}.
\]

Let $\CK^{\mathrm{gen}}(\tU(p,q))$ denote the subgroup of $\CK(\tU(p,q))$ generated by irreducible genuine representations of $\tU(p,q)$.
By \cite[Propositon~7.3]{BMSZ1},
\begin{itemize}
\item if $p+q$ is odd, then
\[\Coh_{\Lambda'}(\CK(\U(p,q))) \cong \cC^1_{p,q,\n_h} \text{ and } \Coh_{\Lambda'}(\CK^{\mathrm{gen}}(\tU(p,q))) \cong \cC^2_{p,q,\n_0};\]
\item if $p+q$ is even, then
\[ \Coh_{\Lambda'}(\CK(\U(p,q))) \cong \cC^2_{p,q,\n_0} \text{ and } \Coh_{\Lambda'}(\CK^{\mathrm{gen}}(\tU(p,q))) \cong \cC^1_{p,q,\n_h}.\]
\end{itemize}

In view of the counting equality of special unipotent representations in terms of the coherent continuation representation \cite[Corollary 2.2]{BMSZ1}, the assertion of Proposition \ref{main3} follows by comparing
the formula of $\Coh_{\Lambda}(\CK(\SU(p,q)))$ in Lemma~\ref{lem:cohSUpq},
with 
\[
\begin{split}
\Coh_{\Lambda'}(\cK(\tU(p,q))) = & \Coh_{\Lambda'}(\CK(\U(p,q))) \oplus \Coh_{\Lambda'}(\CK^{\mathrm{gen}}(\tU(p,q))) \\
= &\cC^1_{p,q,\n_h} \oplus \cC^2_{p,q,\n_0}.
\end{split}
\]
This completes the proof.
\end{proof}

\vspace{5pt}

\begin{ack}
D. Barbasch is supported by NSF grant, Award Number 2000254. J.-J. Ma is supported by the National Natural Science Foundation of China (Grant No. 11701364 and Grant No. 11971305) and Xiamen University Malaysia Research Fund (Grant No. XMUMRF/2022-C9/IMAT/0019).
B. Sun is supported by  National Key R \& D Program of China (No. 2022YFA1005300 and 2020YFA0712600) and New Cornerstone Investigator Program.  C.-B. Zhu is supported by MOE AcRF Tier 1 grant R-146-000-314-114, and
Provost’s Chair grant E-146-000-052-001 in NUS.

C.-B. Zhu is grateful to Max Planck Institute for Mathematics in Bonn, for its warm hospitality and conducive work environment, where he spent the academic year 2022/2023 as a visiting scientist.
\end{ack}

\end{document}